\documentclass[11pt,reqno]{amsart}

\usepackage{amsbsy}
\usepackage{amsfonts}
\usepackage{amsmath}
\usepackage{amsthm}
\usepackage{amssymb}
\usepackage{enumerate}

\usepackage{hyperref}
\usepackage{ytableau}
\usepackage{color}
\usepackage{etoolbox}
\expandafter\patchcmd\csname\string\proof\endcsname
{\normalparindent}{0pt }{}{}
\usepackage{graphicx}

\allowdisplaybreaks

	\newtheorem{thm}{Theorem}[section]
	\newtheorem{lem}[thm]{Lemma}
	\newtheorem{prop}[thm]{Proposition}

\begin{document}

\title[On the distribution of Rudin-Shapiro polynomials]{On the distribution of Rudin-Shapiro polynomials and lacunary walks on $SU(2)$}
\author{Brad Rodgers}
\address{Department of Mathematics, University of Michigan, 530 Church St., Ann Arbor, MI 48109}
\email{rbrad@umich.edu}
\date{}

\begin{abstract}
	We characterize the limiting distribution of Rudin-Shapiro polynomials, showing that, normalized, their values become uniformly distributed in the disc. This resolves conjectures of Saffari and Montgomery. Our proof proceeds by relating the polynomials' distribution to that of a product of weakly dependent random matrices, which we analyze using the representation theory of $SU(2)$. Our approach leads us to a non-commutative analogue of the classical central limit theorem of Salem and Zygmund, which may be of independent interest.
\end{abstract}


\maketitle

\section{Introduction}
\label{sec:1}

The Rudin-Shapiro polynomials are defined inductively as follows: 
$$
P_0(z) = Q_0(z) = 1,
$$ 
and
\begin{equation}
\label{eq:RS1}
P_{k+1}(z) = P_k(z) + z^{2^k} Q_k(z),
\end{equation}
\begin{equation}
\label{eq:RS2}
Q_{k+1}(z) = P_k(z) - z^{2^k}Q_k(z).
\end{equation}
Thus defined, $P_k$ and $Q_k$ are of degree $2^k-1$ and have all coefficients equal to $1$ or $-1$. These polynomials were discovered independently by Golay \cite{Go}, Shapiro \cite{Sh}, and Rudin \cite{Ru} and in addition to having many interesting algebraic and combinatorial properties in their own right \cite{Bri, BrLoMo}, they play an important role in for instance signal processing \cite{laCo}, the study of finite automata \cite{AlSh}, and especially as a source of counterexamples in classical analysis \cite{Mo, KaSa}.

Perhaps most importantly they furnish an example of polynomials $P$ of arbitrarily high degree $N$ with all coefficients $1$ or $-1$ such that $|P(z)| \ll \sqrt{N}$ for all $|z|=1$. Indeed, for $|z|=1$, it may be shown inductively that
$$
|P_k(z)|^2 + |Q_k(z)|^2 = 2^{k+1}.
$$

It remains an open problem whether there exist polynomials $P$ of arbitrarily high degree $N$ with all coefficients $1$ or $-1$ such that $\sqrt{N} \ll |P(z)| \ll \sqrt{N}$ for all $|z|=1$. (See \cite{Od} for recent numerical evidence that such polynomials exist.\footnote{If the coefficients are instead only restricted to be complex unimodular, such polynomials do indeed exist and in fact one can ensure $|P(z)|\sim \sqrt{N}$ uniformly; see \cite{Ka,Be,BoBo}.}) 
As a consequence of results proved in this paper we will see that Rudin-Shapiro polynomials do not satisfy a lower bound of this sort -- this has always been quite evident numerically, though there does not seem to be a proof of this fact already in the literature\footnote{The polynomials $P_k$ satisfy the recursion $P_{k+2}(z) = (1-z^{2^{k+1}})P_{k+1}(z) + 2 z^{2^{k+1}}P_k(z)$ (see \cite{BrLoMo}), and so it follows inductively that $P_k(-1) = 0$ for odd $k$. In this way, H. Montgomery has pointed out, one may observe for odd $k$ that $P_k(z)$ does not satisfy lower bound of this sort, but it is not clear that a similarly simple proof exists for even $k$.} -- and it is natural moreover to ask how far they deviate from such a lower bound. Investigations related to this question date at least back to Saffari in 1980, who let $\omega$ be a random variable distributed uniformly on the unit circle $|\omega|=1$ and asked about the radial distribution of $P_k(\omega)$. Saffari conjectured that for any $n\geq 0$,
\begin{equation}
\label{eq:saffari_moments}
\mathbb{E}\, \Big| \frac{P_k(\omega)}{\sqrt{2^{k+1}}}\Big|^{2n} \sim \frac{1}{n+1}
\end{equation}
as $k\rightarrow\infty$ or equivalently that
\begin{equation}
\label{eq:saffari_dist}
\mathbb{P}\Big(\,\Big| \frac{P_k(\omega)}{\sqrt{2^{k+1}}}\Big|^2 \in [\alpha,\beta]\Big) \sim \beta-\alpha
\end{equation}
as $k\rightarrow\infty$, for $0\leq \alpha < \beta \leq 1$. Saffari evidently did not publish this conjecture himself, and it first appeared in print in work of Doche and Habsieger \cite{DoHa}, who verified \eqref{eq:saffari_moments} for $n\leq 26$. Erd\'elyi \cite{Er} has considered closely related questions involving the Mahler measure of Rudin-Shapiro polynomials, and obtains absolute upper bounds for moments.

In this note we resolve Saffari's conjecture in general:

\begin{thm}[Saffari's Conjecture]
\label{thm:saffari}
For $\omega$ a random variable uniformly distributed on the unit circle, \eqref{eq:saffari_moments} and \eqref{eq:saffari_dist} are true.
\end{thm}

In fact with more work we will be able to resolve a more recent conjecture of Montgomery \cite{Mo, Mo2} which generalizes Saffari's: we show that $P_k(\omega)/\sqrt{2^{k+1}}$ tends towards uniform distribution in the unit disc $D = \{z \in \mathbb{C}:\, |z|\leq 1\}$, as $k\rightarrow\infty$:

\begin{thm} [Montgomery's Conjecture]
\label{thm:montgomery}
For any rectangle $E\subseteq D$,
$$
\mathbb{P}\Big( \frac{P_k(\omega)}{\sqrt{2^{k+1}}} \in E \Big) \sim \frac{1}{\pi} |E|,
$$
as $k\rightarrow\infty$.
\end{thm}


Our proofs of both results proceed by characterizing the distribution of a product of weakly dependent random matrices. Note that from the inductive definition \eqref{eq:RS1} and \eqref{eq:RS2}, we have
$$
\begin{pmatrix} P_k(z) \\ Q_k(z) \end{pmatrix} = \begin{pmatrix} 1 & z^{2^k} \\ 1 & - z^{2^k} \end{pmatrix} \begin{pmatrix} 1 & z^{2^{k-1}} \\ 1 & -z^{2^{k-1}} \end{pmatrix} \cdots \begin{pmatrix} 1 & z \\ 1 & -z \end{pmatrix} \begin{pmatrix} 1 \\ 0 \end{pmatrix}.
$$
We will work with a normalized version of the matrices above. Define
$$
g(z):= \frac{1}{\sqrt{2}} \begin{pmatrix} i z^{-1} & i z \\ i z^{-1} & - iz \end{pmatrix},
$$
which is an element of $SU(2)$ for $|z|=1$. We have
\begin{equation}
\label{link}
(i^{k+1} z^{2^{k+1}-1})\begin{pmatrix} P_k(z^2)/\sqrt{2^{k+1}} \\ Q_k(z^2)/\sqrt{2^{k+1}} \end{pmatrix} = g(z^{2^k})\cdots g(z) \begin{pmatrix} 1 \\ 0 \end{pmatrix}.
\end{equation}
We note $P_k(\omega)$ has the same distribution as $P_k(\omega^2)$, so Saffari's conjecture
will directly follow from
\begin{thm}[Equidistribution in $SU(2)$]
\label{thm:matrix_product1}
As $k\rightarrow\infty$, the distribution of $g(\omega^{2^k})\cdots g(\omega)$ tends to Haar measure on $SU(2)$.	
\end{thm}
Indeed, this shows that the distribution of $|P_k(\omega)/\sqrt{2^{k+1}}|^2$ tends to that of
$$
\Big| \Big\langle \begin{pmatrix} 1 \\ 0 \end{pmatrix}, g \begin{pmatrix} 1 \\ 0 \end{pmatrix} \Big\rangle \Big|^2,
$$
for $g\in SU(2)$ chosen according to Haar measure. We leave it to the reader to verify that this implies \eqref{eq:saffari_dist} (using for instance an Euler angle parameterization of $g$, see e.g. \cite[Prop. 7.4.1]{Fa}).

Returning our attention to the representation \eqref{link}, we shall also show
\begin{lem}
\label{lem:independence}
The random variables $\omega^{2^{k+1}-1}$ and $g(\omega^{2^k})\cdots g(\omega)$ become independent as $k\rightarrow\infty$. 
\end{lem}
As a consequence we immediately obtain Montgomery's conjecture and more generally
\begin{thm}[Equidistribution in $U(2)$]
\label{thm:matrix_product2}
Let
$$
G(\omega) := \frac{1}{\sqrt{2}} \begin{pmatrix} 1 & \omega \\ 1 & -\omega \end{pmatrix}.
$$
As $k\rightarrow\infty$, the distribution of $G(\omega^{2^k})\cdots G(\omega)$ tends to Haar measure on $U(2)$.
\end{thm}

Montgomery's conjecture is thus verified by noting the distribution of $P_k(\omega)/\sqrt{2^{k+1}}$ tends to that of
$$
\Big\langle \begin{pmatrix} 1 \\ 0 \end{pmatrix}, G \begin{pmatrix} 1 \\ 0 \end{pmatrix} \Big\rangle,
$$
where $G \in U(2)$ is chosen according to Haar measure. This may again be analyzed using Euler angles. Alternatively, perhaps more simply, Montgomery's conjecture follows from the Lemma and Theorem \ref{thm:saffari} by noting that the distribution of $P_k(\omega)/\sqrt{2^{k+1}}$ tends to that of $\zeta \cdot |P_k(\omega)/\sqrt{2^{k+1}}|$ where $\zeta$ is an independent random variable uniformly distributed on the unit circle.\footnote{This may be compared to the results proved in \cite{Er2} for the class of \emph{flat polynomials}.}

The moral of Theorems \ref{thm:matrix_product1} and \ref{thm:matrix_product2} should be clear; it is that the sequence of random variables $\omega, \omega^2, ..., \omega^{2^k}$ resemble i.i.d. variables $\omega_0, \omega_1, ..., \omega_k$ and therefore the matrix product in Theorem \ref{thm:matrix_product1} for instance has a distribution close to that of $g(\omega_0)\cdots g(\omega_k)$. But this latter matrix product is just a random walk on the compact group $SU(2)$, which is known to equidistribute unless certain obvious obstacles (support on a proper closed subgroup or support on a coset of a proper closed subgroup \cite{KaIt}) arise, which in this case they do not. These theorems may thus be thought of as a non-commutative version of the well-known central limit theorem of Salem and Zygmund for lacunary trigonometric series \cite{SaZy}.

There is a sense in which this analogy to statistical independence may be misleading and this is discussed in the final section of this paper. In any case it is not by a comparison to a product of i.i.d. variables that we prove Theorem \ref{thm:matrix_product1} and Lemma \ref{lem:independence}. Instead we return to earlier principles and make use of the representation theory of $SU(2)$.

We note that very recently and independently, D. Zeilberger  \cite{Ze} has sketched a different possible approach to Saffari's conjecture from the one we take. Zeilberger's sketch is contingent on the algorithmic proof of certain identities which have been found empirically and also a claim called \emph{the small change hypothesis} which is used to bound error terms. This hypothesis does not necessarily follow algorithmically at the present moment, but it might after an explicit identification of certain terms that arise in his sketch. \cite{Ze} also computes the asymptotics of moments $\mathbb{E} \,[\overline{P_k(\omega)}]^n [P_k(\omega)]^m$ for small $n$ and $m$, with an output consistent with Montgomery's conjecture.

\section{Equidistribution in $SU(2)$: a proof of Theorem \ref{thm:matrix_product1}}
\label{sec:2}

This section is devoted to a proof of Theorem \ref{thm:matrix_product1} regarding equidistribution in $SU(2)$. Our proof of Lemma \ref{lem:independence} and thus Theorem \ref{thm:matrix_product2}, which generalizes the result to $U(2)$, uses similar methods but adds an additional layer of complexity and will come in the next section.

Our approach will be to demonstrate for every nontrivial irreducible representation $\pi$ of $SU(2)$ that\footnote{By \eqref{irreps_to_0} we mean that the expectation of all matrix coefficients of the representations tend to $0$; that is for any nontrivial irreducible representation $(\pi, \mathbb{C}^n)$, we have $\mathbb{E}\,\langle u, \pi(g(\omega^{2^k}))\cdots \pi(g(\omega))v\rangle \rightarrow 0$ for all vectors $u, v \in \mathbb{C}^n$.}
\begin{equation}
\label{irreps_to_0}
\mathbb{E}\, \pi(g(\omega^{2^k}) \cdots g(\omega) )= \mathbb{E}\,\pi(g(\omega^{2^k}))\cdots \pi(g(\omega)) \rightarrow 0.
\end{equation}
By a well-known criterion, this is necessary and sufficient to prove the theorem. Indeed, in general we have
\begin{thm}
\label{Levy_generalization}
For $H$ a compact Lie group and $h_1, h_2, h_3,...$ a sequence of $H$-valued random elements, the distributions of $h_k$ tend to Haar measure if and only if for every nontrivial irreducible representation $\pi$ of $H$,
\begin{equation}
\label{exp_to_zero}
\mathbb{E}\, \pi(h_k) \rightarrow 0.
\end{equation}
\end{thm}
Theorem \ref{Levy_generalization} is a corollary of Theorem 4.2.5 of \cite{Ap}. Theorem 4.2.5 there is a slightly more general result characterizing arbitrary limiting distributions; the case that the limiting distribution of the random elements is Haar measure follows from Example 2 of section 4.2 in the same source.

This approach of demonstrating \eqref{irreps_to_0} we note, is very much just a variant of the moment method in this context, and fortunately the representation theory of $SU(2)$ is elegant and well-understood (see, e.g. \cite{Fa}, \cite{Vi}, or \cite{AnAsRo}). We recall the necessary facts here.

Any matrix of $SU(2)$ has the form
$$
g = \begin{pmatrix} \alpha & \beta \\ - \overline{\beta} & \overline{\alpha} \end{pmatrix},
$$
with $|\alpha|^2 + |\beta|^2  = 1.$ Nontrivial irreducible representations of the matrix $g$ are parameterized by semi-integers $\ell = 1/2, 1, 3/2, 2, ...$ and consist of the matrices with entries
$$
t^\ell_{m,n}(g) = \sqrt{\frac{(\ell-m)! (\ell+m)!}{(\ell-n)!(\ell+n)!}} \int_\Gamma (\alpha z + \gamma)^{\ell-n} (\beta z + \delta)^{\ell+n} z^{m-\ell} \frac{dz}{2\pi i z}
$$
where $m,n \in \{ -\ell, -\ell+1,..., \ell\}$, and the contour $\Gamma$ is the unit circle. (Here the range of $m$ and $n$ is such that $\ell-n$ and $\ell+n$ are integers.) Note that
\begin{equation}
\label{eq:t_def}
t^\ell_{m,n}(g(\omega))  = \mathfrak{\tau}^\ell_{m,n}\, \omega^{2n}
\end{equation}
with
\begin{equation*}
\tau^\ell_{m,n}:= \sqrt{\frac{(\ell-m)! (\ell+m)!}{(\ell-n)!(\ell+n)!}} \frac{i^{2\ell}}{2^\ell} \int_\Gamma (z + 1)^{\ell-n} (z -1)^{\ell+n} z^{m-\ell} \frac{dz}{2\pi i z}.
\end{equation*}
Note that $\tau^\ell$ is itself a unitary matrix, corresponding to the representation of the matrix $g(1)$.

If $\ell = 1/2, 3/2, 5/2,...$ then it is transparent that 
$$
\mathbb{E}\, t^\ell(g(\omega^{2^k})) \cdots t^{\ell}(g(\omega)) = 0,
$$
for all $n\geq 0$, since every entry of the matrix product of which we are taking the expectation will be a Laurent polynomial in $\omega$ with only odd powers appearing.

On the other hand, for $\ell = 1,2,...$ nothing so simple is true. Each entry of $t^\ell(g(\omega^{2^k}))\cdots t^\ell(g(\omega))$ will be a Laurent polynomial in $\omega$ with exclusively even powers however, so at least we have
$$
\mathbb{E}\, t^\ell(g(\omega^{2^k}))\cdots t^\ell(g(\omega)) = \mathbb{E}\, t^\ell(g(\omega^{2^{k-1}}))\cdots t^\ell(g(\omega^{1/2})). 
$$
We use the notation:
$$
T^\ell(\omega) := \tau^\ell \begin{pmatrix} \omega^{-\ell} & & & \\ & \omega^{-\ell+1} & & & \\ & & \ddots & \\ & & & \omega^\ell \end{pmatrix} \quad \textrm{so that}\quad T^\ell(\omega)= t^\ell(g(\omega^{1/2})).
$$
In this notation, we must show $\mathbb{E}\, T^\ell(\omega^{2^k}) \cdots T^\ell(\omega) \rightarrow 0$ to prove the theorem; in turn to show this we need only show that all constant coefficients in $T^\ell(\omega^{2^k})\cdots T^\ell(\omega)$ go to $0$. To this end we note the following two facts:
\begin{enumerate}
	\item The matrix entries of $T^\ell(\omega^{2^k})\cdots T^\ell(\omega)$ will be Laurent polynomials in $\omega$ lying in the span of $\{ \omega^{(2^{k+1}-1)\ell}, ..., \omega^{-(2^{k+1}-1)\ell}\}.$
	
	\item To find a coefficient of $\omega^{\nu 2^{k+1}}$ for $\nu \in \mathbb{Z}$ (and in particular the constant coefficient) of a matrix entry of $$T^\ell(\omega^{2^k})\cdots T^\ell(\omega) = T^\ell(\omega^{2^k}) \times \Big(T^\ell(\omega^{2^{k-1}})\cdots T^\ell(\omega)\Big)$$ we need only keep track of the coefficients of $\omega^{\mu 2^k}$ for $\mu \in \mathbb{Z}$ in the entries of $T^\ell(\omega^{2^{k-1}})\cdots T^\ell(\omega).$
\end{enumerate}

We let $P$ be the space of Laurent polynomials in $\omega$ (with complex coefficients) and define an operator $S^\ell$ on the product space $P^{2\ell+1}$ as follows: for 
$$
A = \begin{bmatrix} A_{-\ell}(\omega), & \cdots & ,A_\ell(\omega) \end{bmatrix}^{T} \in P^{2\ell+1},
$$
if
$$
T^\ell(\omega) A = \begin{bmatrix} \sum_{j\in \mathbb{Z}} \beta_{-\ell}(j) \omega^j, & \cdots & ,\sum_{j \in \mathbb{Z}} \beta_\ell(j) \omega^j \end{bmatrix}^T
$$
where the coefficients $\beta_h(j)$are defined by this relation, we define
$$
S^\ell A := \begin{bmatrix} \sum_{j \in \mathbb{Z}} \beta_{-\ell}(2j) \omega^j, & \cdots & ,\sum_{j \in \mathbb{Z}} \beta_\ell(2j) \omega^j \end{bmatrix}^T.
$$

We note, in view of fact (ii) above, that for arbitrary $v \in \mathbb{C}^{2\ell+1}$
\begin{equation}
\label{eq:expect}
\mathbb{E}\, T^\ell(\omega^{2^k}) \cdots T^\ell(\omega) v = \mathbb{E}\, (S^\ell)^{k+1} v.
\end{equation}
Moreover, define the space of Laurent polynomials $P_\ell := \textrm{span}_\mathbb{C}\{ \omega^{-(\ell-1)},...,\omega^{\ell-1}\}$. Note that $S^\ell$ is a linear operator mapping the finite dimensional complex vector space $(P_\ell)^{2\ell+1}$ into itself. We let $S_\ell$ be the operator $S^\ell$ restricted to $(P_\ell)^{2\ell+1}$, and rewrite the vector in \eqref{eq:expect} as
\begin{equation}
\label{eq:expect2}
\mathbb{E}\, (S_\ell)^{k+1} v.
\end{equation}
If we show that for all $\ell$ the spectral radius of $S_\ell$ is less than $1$, then $(S_\ell)^{k+1}\rightarrow 0$ as $k\rightarrow\infty$, and \eqref{eq:expect2} likewise will tend to $0$ for all vectors $v$ which implies the theorem. We let $\rho(\cdot)$ denote the spectral radius of an operator; the remainder of this section is devoted to a proof that $\rho(S_\ell) < 1$.

We make the following observation:

\begin{prop}
\label{prop:trivial_bound_and_if_}
In the above notation, $\rho(S_\ell) \leq 1$. Moreover, if $\rho(S_\ell)=1$, then there must exist non-zero $A \in (P_\ell)^{2\ell+1}$ such that
$$
T^\ell(\omega) A(\omega) = c A(\omega^2).
$$
\end{prop}

\begin{proof}
	From the definition if $S_\ell$ has an eigenvalue $c$, there must exist non-zero $A \in (P_\ell)^{2\ell+1}$ and $B \in P^{2\ell+1}$ with
	$$
	T^\ell(\omega) A(\omega) = c A(\omega^2) + \omega B(\omega^2),
	$$
	by separating the polynomial vector on the left hand side into odd and even powers. But because $T^\ell(\omega)$ is unitary and $A(\omega^2)$ and $\omega B(\omega^2)$ have complementary powers, this implies that
	\begin{align}
	\label{eigen_identity}
	\notag \mathbb{E}\, \|A(\omega)\|_{\ell^2}^2 &= |c|^2 \cdot \mathbb{E} \,\|A(\omega^2)\|_{\ell^2}^2 + \mathbb{E}\,\|B(\omega^2)\|_{\ell^2}^2 \\
	&= |c|^2 \cdot \mathbb{E} \,\|A(\omega)\|_{\ell^2}^2 + \mathbb{E}\,\|B(\omega)\|_{\ell^2}^2.
	\end{align}
	This obviously implies $|c| \leq 1$, so $\rho(S_\ell) \leq 1$. Moreover, if $\rho(S_\ell)=1$, then $|c|=1$, and \eqref{eigen_identity} implies that $B = 0$.
\end{proof}

We now suppose that it is the case that $\rho(S_\ell) = 1$, aiming at a proof by contradiction. By the proposition, we will obtain a contradiction if we show that for $|c|=1$, the only $A \in (P_\ell)^{2\ell+1}$ satisfying $T^\ell(\omega) A(\omega) = c A(\omega^2)$ is $A = 0$. Labeling the coefficients of $A$, we are looking to find numbers $\alpha_h(j)$ such that
$$
\tau^\ell \begin{pmatrix} \omega^{-\ell} & & & \\ & \omega^{-(\ell-1)} & & & \\ & & \ddots & \\ & & & \omega^\ell \end{pmatrix} \begin{pmatrix} \sum_{j = -(\ell-1)}^{\ell-1} \alpha_{-\ell}(j) \omega^j \\ \vdots \\ \sum_{j = -(\ell-1)}^{\ell-1} \alpha_{\ell}(j) \omega^j \end{pmatrix} = c \begin{pmatrix} \sum_{j = -(\ell-1)}^{\ell-1} \alpha_{-\ell}(j) \omega^{2j} \\ \vdots \\ \sum_{j = -(\ell-1)}^{\ell-1} \alpha_{\ell}(j) \omega^{2j} \end{pmatrix}
$$
By equating the coefficient of $\omega^\nu$ in each coordinate of the two column vectors above we obtain a homogeneous system of linear equations in the $\alpha_h(j)$. We give a proof that the only solution of this system of equations is $\alpha_h(j)=0$ for all $h,j$.

Let $\mathcal{L}(\nu)$ be the linear equation in the coefficients $\alpha_h(j)$ that results from examining the coefficient of $\omega^\nu$ in the system above. Meaningful information is got from $\mathcal{L}(\nu)$ for $-2\ell+1 \leq \nu \leq 2\ell-1$; in this range $\mathcal{L}(\nu)$ becomes
$$
\tau_\ell \begin{pmatrix} \alpha_{-\ell}(\nu+\ell) \\ \alpha_{-\ell+1}(\nu+\ell-1) \\ \vdots \\ \alpha_\ell(\nu-\ell) \end{pmatrix} = c \begin{pmatrix} \alpha_{-\ell}(\nu/2) \\ \alpha_{-\ell+1}(\nu/2) \\ \vdots \\ \alpha_\ell(\nu/2) \end{pmatrix}
$$
with the convention that $\alpha_h(j)=0$ if $j$ is not an integer or $|j| > \ell-1$.

The information we need about the matrix $\tau^\ell$ in solving this system is not very special. The properties we need are contained in
\begin{prop}
\label{prop:special_matrix}
For $\ell = 1, 2,...$ the $(2\ell+1) \times (2\ell+1)$ matrix $\tau^\ell$ satisfies the following properties
\begin{enumerate}[a)]
	\item $\tau^\ell$ is invertible.
	\item $|\tau^\ell_{00}| < 1$.
	\item The matrix $\tau^\ell$ is such that if any one of the linear equations below hold,
	\begingroup\makeatletter\def\f@size{9}\check@mathfonts
	\begin{multline*}
	\tau^\ell \begin{pmatrix} \beta_{-\ell} \\ \vdots \\ \beta_{-1} \\ 0 \\ 0 \\ \vdots \\ 0 \end{pmatrix} = \begin{pmatrix} \gamma_0 \\ 0 \\ \gamma_1 \\ 0 \\ \gamma_2 \\ \vdots \\ \gamma_{\ell} \end{pmatrix}  
	\quad \textrm{or} \quad 
	\tau^\ell \begin{pmatrix} \beta_{-\ell} \\ \vdots \\ \beta_{-1} \\ 0 \\ 0 \\ \vdots \\ 0 \end{pmatrix} = \begin{pmatrix} 0 \\ \gamma_0 \\ 0 \\ \gamma_1 \\ 0 \\ \vdots \\ 0 \end{pmatrix}
	\\ \quad \textrm{or} \quad 
	\tau^\ell\begin{pmatrix} 0 \\ \vdots \\ 0 \\ 0 \\ \beta_1 \\ \vdots \\ \beta_\ell \end{pmatrix} = \begin{pmatrix} \gamma_0 \\ 0 \\ \gamma_1 \\ 0 \\ \gamma_2 \\ \vdots \\ \gamma_\ell \end{pmatrix} 
	\quad \textrm{or} \quad 
	\tau^\ell \begin{pmatrix} 0 \\ \vdots \\ 0 \\ 0 \\ \beta_1 \\ \vdots \\ \beta_\ell \end{pmatrix} = \begin{pmatrix} 0 \\ \gamma_0 \\ 0 \\ \gamma_1 \\ 0 \\ \vdots \\ 0 \end{pmatrix},	
	\end{multline*}
	\endgroup
	then $\beta_i = 0$ and $\gamma_i = 0$ for all $i$.	
\end{enumerate}
\end{prop}
\begin{proof}
Of these properties, a) is true simply because $\tau^\ell$ is unitary and b) is straightforward to verify. For c), we will show that the first of these linear equations implies $\beta_i, \gamma_i = 0$ for all $i$, with the others being handled similarly. We make the linear change of variables
$$
\beta_{-n}':= \frac{(-1)^\ell}{2^\ell} \frac{1}{\sqrt{(\ell-n)!(\ell+n)!}}\beta_{-n}, \quad \gamma_i' := \frac{1}{\sqrt{(2\ell-i)!(2i)!}}\gamma_i,
$$
and note from the definition of $\tau_\ell$ that the first linear equation is equivalent to the polynomial identity
\begin{multline*}
\beta_{-\ell}' (z+1)^{2\ell} + \beta_{-\ell+1}' (z+1)^{2\ell-1} (z-1) + \beta_{-1}'(z+1)^{\ell+1}(z-1)^{\ell-1} \\
= \gamma_0' z^{2\ell} + \gamma_1' z^{2\ell-2} + \cdots + \gamma_\ell'.
\end{multline*}
We let $P(z):= \gamma_0' z^\ell + \gamma_1' z^{\ell-1} + \cdots + \gamma_\ell'$. Then $(z+1)^{\ell+1}| P(z^2)$. But $P((-z)^2) = P(z^2)$, so $(z-1)^{\ell+1}|P(z^2)$ also. Since $(z+1)^{\ell+1}$ and $(z-1)^{\ell+1}$ are coprime, this implies by multiplying the two factors together that $(z^2-1)^{\ell+1}|P(z^2)$ and therefore that $(z-1)^{\ell+1}|P(z)$. Yet $\deg P(z) \leq \ell$, so we must have $P = 0$. Because $P=0$, we have $\gamma_i=0$ for all $i$ and $\beta_i = 0$ for all $i$ as well.
\end{proof}

Using this proposition, we examine the linear equations $\mathcal{L}(-2\ell+1)$, $\mathcal{L}(-2\ell+3)$,..., $\mathcal{L}(2\ell-1)$. 
By writing out the resulting equations and utilizing the invertibility of $\tau_\ell$, one sees that the $(2\ell+1)\times (2\ell-1)$ matrix
$$
\alpha:= \begin{pmatrix}
\alpha_{-\ell}(-\ell+1) & \alpha_{-\ell}(-\ell+2) & \cdots & \alpha_{-\ell}(\ell-1) \\
\alpha_{-\ell+1}(-\ell+1) &\ddots & & \vdots \\
\vdots & & & \vdots \\
\alpha_\ell(-\ell+1) & \cdots & \cdots & \alpha_\ell(\ell-1) \end{pmatrix}
$$
with $4\ell-1$ skew-diagonals in total, has entirely $0$ entries along its alternating skew-diagonals (that is, from top left, its $1^\mathrm{st}$, $3^\mathrm{rd}$,..., $(4\ell-1)^\mathrm{th}$ skew-diagonals).

We now consider $\mathcal{L}(-2\ell+2)$, which using the information we have just obtained reduces to the equation
$$
\tau^\ell \begin{pmatrix} \alpha_{-\ell}(-\ell+2) \\ \alpha_{-\ell+1}(-\ell+1) \\ 0 \\ \vdots \\ 0 \end{pmatrix} = c \begin{pmatrix} 0 \\ \alpha_{-\ell+1}(-\ell+1) \\ 0 \\ \alpha_{-\ell+3}(-\ell+1) \\ \vdots \\ 0 \end{pmatrix}
$$
By part c) of Proposition \ref{prop:special_matrix}, both vectors in the above equation must be $0$, so that the second skew-diagonal of $\alpha$ also consists of $0$ entries, as does the first (leftmost) column.

In turn, updating $\alpha$ to reflect the fact that its first column is $0$, the linear equation $\mathcal{L}(-2\ell+4)$ becomes
$$
\tau^\ell \begin{pmatrix} \alpha_{-\ell}(-\ell+4) \\ \alpha_{-\ell+1}(-\ell+3) \\ \alpha_{-\ell+2}(-\ell+2) \\ 0 \\ \vdots \\ 0 \end{pmatrix} = c \begin{pmatrix} \alpha_{-\ell}(-\ell+2) \\ 0 \\ \alpha_{-\ell+2}(-\ell+2) \\ 0 \\ \vdots \\ \alpha_{\ell}(-\ell+2) \end{pmatrix}
$$
Again by part c) of the proposition this implies both vectors above are $0$, so that the fourth skew-diagonal and second column of $\alpha$ consists of $0$ entries.

Continuing on in this fashion, after $\mathcal{L}(-2\ell+2)$ and $\mathcal{L}(-2\ell+4)$, we may consider $\mathcal{L}(-2\ell+6),...,\mathcal{L}(-2)$, with $\mathcal{L}(-2\ell+2j)$ becoming one of the two equations:
$$
\tau^\ell \begin{pmatrix} \alpha_{-\ell}(-\ell+2j) \\ \vdots \\ \alpha_{-\ell+j}(-\ell+j) \\ 0 \\ \vdots \\ 0 \end{pmatrix} = c \begin{pmatrix} 0 \\ \alpha_{-\ell+1}(-\ell+j) \\ 0 \\ \alpha_{-\ell+3}(-\ell+j) \\ \vdots \\ 0 \end{pmatrix} \quad \mathrm{or} \quad c \begin{pmatrix} \alpha_{-\ell}(-\ell+j) \\ 0 \\ \alpha_{-\ell+2}(-\ell+j) \\ 0 \\ \vdots \\ \alpha_\ell(-\ell+j) \end{pmatrix}.
$$
In either case, part c) of the proposition implies that the vectors above are $0$. Taken all together we see that the first $\ell-1$ columns of $\alpha$ consist of $0$ entries, and likewise the $2^\mathrm{nd}$, $4^\mathrm{th}$, ..., $(2\ell-2)^{\mathrm{th}}$ skew-diagonals consist of $0$ entries as well.

This same argument may be run in the reverse direction by considering $\mathcal{L}(2\ell-2))$, $\mathcal{L}(2\ell-4)),...,\mathcal{L}(2)$, with the result that the last $\ell-1$ columns of $\alpha$ consist of $0$ entries, as do the $(4\ell-2)^\mathrm{th}$, $(4\ell-4)^\mathrm{th}$,..., $(2\ell+2)^\mathrm{th}$ skew-diagonals.

We have therefore shown that all skew-diagonals but the middle $2\ell^\mathrm{th}$ (out of $4\ell-1$) consist of $0$ entries. Likewise all columns, except for perhaps the middle $\ell^\mathrm{th}$ (out of $2\ell-1$) consist of $0$ entries. By combining the two pieces of information, one sees that that all entries of the matrix $\alpha$ except perhaps $\alpha_0(0)$ are equal to $0$. But then $\mathcal{L}(0)$ reduces to the statement that $\tau^\ell_{00}\, \alpha_0(0) = c\, \alpha_0(0)$, which by part b) of Proposition \ref{prop:special_matrix} is impossible for $|c|=1$ unless $\alpha_0(0) = 0$ as well.

Hence all eigenvalues of the operator $S_\ell$ must be smaller in modulus than $1$ as claimed, and Theorem \ref{thm:matrix_product1} follows.

\section{Independence of determinant and normalized entries: A proof of Lemma \ref{lem:independence}}
\label{sec:3}

In this section we prove Lemma \ref{lem:independence}. By virtue of the Peter-Weyl theorem, to prove that $\omega^{2^{k+1}-1}$ and $g(\omega^{2^k})\cdots g(\omega)$ become independent as $k\rightarrow\infty$, it will be sufficient to show that
\begin{equation}
\label{eq:ind_moments}
\mathbb{E}\, (\omega^{2^{k+1}-1})^\lambda \pi(g(\omega^{2^k})\cdots g(\omega)) = \mathbb{E}\, \omega^{\lambda 2^k} \pi(g(\omega^{2^k}))\cdots \omega^\lambda \pi(g(\omega)) \rightarrow 0
\end{equation}
as $k\rightarrow\infty$, for any fixed $\lambda \in \mathbb{Z}$ and irreducible representation $\pi$ of $SU(2)$. Our proof will mimic the proof the Theorem \ref{thm:matrix_product1} in the last section. We break the proof into two cases. 

\textbf{Case 1:} $\lambda$ is even. In this case, as before, for $\pi = t^\ell$ with $\ell = 1/2, 3/2, ...$, it is clear that \eqref{eq:ind_moments} is true since each matrix entry of which we are taking the expectation will be a Laurent polynomial in $\omega$ with only odd powers. So it will suffice in this case to consider $\pi = t^\ell$ with $\ell = 1,2,...$.

\textbf{Case 2:} $\lambda$ is odd. In this case, for $\pi = t^\ell$ with $\ell=1,2,...$ it is clear that \eqref{eq:ind_moments} is true, for the same reason as in case 1. Hence it will suffice to consider $\pi = t^\ell$ with $\ell = 1/2, 3/2,...$.

We analyze \textbf{case 2} first. For $\ell=1/2,3/2,...$ let $T^\ell(\omega) = t^\ell(g(\omega^{1/2}))$ and observe that the problem reduces to showing that
\begin{equation}
\label{eq:tends_to_0}
\mathbb{E} \,\omega^{(\lambda/2)2^k} T^\ell(\omega^{2^k})\cdots \omega^{\lambda/2} T^\ell(\omega) \rightarrow 0.
\end{equation}
Note that
$$
\omega^{\lambda/2} T^\ell(\omega)  = \tau^\ell \begin{pmatrix} 
\omega^{\lambda/2-\ell} & & & \\
& \omega^{\lambda/2-\ell+1} & & \\
& & \ddots & \\
& & & \omega^{\lambda/2+\ell}\end{pmatrix}
$$
so although $\lambda/2$ and $\ell$ are half-integers, all powers of $\omega$ appearing above are integers. 

We will need a result similar to Proposition \ref{prop:special_matrix} for half-integer\footnote{Note that this implies $2\ell+1$ is now an even integer.} $\ell$.
\begin{prop}
\label{prop:special_matrix2}
For $\ell = 1/2, 3/2,...$ the $(2\ell+1)\times (2\ell+1)$ matrix $\tau^\ell$ satisfies the following properties
\begin{enumerate}[a)]
\item $\tau^\ell$ is invertible.
\item $|\tau^\ell_{-\ell,-\ell}| < 1$ and $|\tau^\ell_{\ell,\ell}| < 1$.
\item The matrix $\tilde{t}^\ell$ is such that if any one of the linear equations below hold,
\begingroup\makeatletter\def\f@size{10}\check@mathfonts
\begin{multline*}
	\tau^\ell \begin{pmatrix} \beta_{-\ell} \\ \vdots \\ \beta_{-1/2} \\ 0 \\ \vdots \\ 0 \end{pmatrix} = \begin{pmatrix} \gamma_0 \\ 0 \\ \gamma_1 \\ 0 \\ \vdots \\ 0 \end{pmatrix}  
	\quad \textrm{or} \quad 
	\tau^\ell \begin{pmatrix} \beta_{-\ell} \\ \vdots \\ \beta_{-1/2} \\ 0 \\ \vdots \\ 0 \end{pmatrix} = \begin{pmatrix} 0 \\ \gamma_0 \\ 0 \\ \gamma_1 \\ \vdots \\ \gamma_{\ell-1/2} \end{pmatrix}
	\\
	\quad \textrm{or} \quad 
	\tau_\ell\begin{pmatrix} 0 \\ \vdots \\ 0 \\ \beta_{1/2} \\ \vdots \\ \beta_\ell \end{pmatrix} = \begin{pmatrix} \gamma_0 \\ 0 \\ \gamma_1 \\ 0 \\ \vdots \\ 0 \end{pmatrix}   
	\quad \textrm{or} \quad 
	\tau^\ell \begin{pmatrix} 0 \\ \vdots  \\ 0 \\ \beta_{1/2} \\ \vdots \\ \beta_\ell \end{pmatrix} = \begin{pmatrix} 0 \\ \gamma_0 \\ 0 \\ \gamma_1 \\ \vdots \\ \gamma_{\ell-1/2} \end{pmatrix},	
\end{multline*}
\endgroup
then $\beta_i = 0$ and $\gamma_i=0$ for all $i$.
\end{enumerate}
\end{prop}
\begin{proof} 
a) and b) are of course straightforward, while c) follows by mimicking the argument in Proposition \ref{prop:special_matrix}.
\end{proof}

In demonstrating \eqref{eq:tends_to_0}, we consider four possibilities:

\noindent \textbf{1)} If $0 \notin [\lambda/2-\ell, \lambda/2+\ell]$, then the left hand side of \eqref{eq:tends_to_0} is $0$ for all $k$.

\noindent \textbf{2)} If $0 = \lambda/2-\ell$, then all entries of $\omega^{\lambda/2}T^\ell(\omega)$ are polynomials in $\omega$ (i.e. no negative powers appear) and by inspection of matrix entries,
$$
\mathbb{E} \,\omega^{(\lambda/2)2^k} T^\ell(\omega^{2^k})\cdots \omega^{\lambda/2} T^\ell(\omega) = (t^\ell_{-\ell,-\ell})^k \begin{pmatrix} t^\ell_{-\ell,-\ell} & 0 & \cdots \\ 
\vdots & \vdots & \ddots \\
t^\ell_{\ell,-\ell} & 0 & \cdots \end{pmatrix} \rightarrow 0,
$$
by part b) of the proposition.

\noindent \textbf{3)} If $0 = \lambda/2+\ell$, then all entries of $\omega^{\lambda/2} T^\ell(\omega)$ are polynomials in $\omega^{-1}$ (i.e. no positive powers appear) and in this case,
$$
\mathbb{E} \,\omega^{(\lambda/2)2^k} T^\ell(\omega^{2^k})\cdots \omega^{\lambda/2} T^\ell(\omega) = (t^\ell_{\ell,\ell})^k \begin{pmatrix} \cdots & 0 & t^\ell_{-\ell,\ell} \\ 
\ddots & \vdots & \vdots \\
 \cdots & 0 & t^\ell_{\ell,\ell} \end{pmatrix} \rightarrow 0,
$$
again by part b) of the proposition.

\noindent \textbf{4)} Finally, if $0 \in (\lambda/2-\ell, \lambda/2+\ell)$, then we proceed along lines similar to the last section. We note much as before
\begin{enumerate}
\item The matrix entries of $\omega^{(\lambda/2)2^k} T^\ell(\omega^{2^k})\cdots \omega^{\lambda/2} T^\ell(\omega)$ will be Laurent polynomials in $\omega$ lying in the span of $\{\omega^{(2^{k+1}-1)(\lambda/2+\ell)},..., \omega^{(2^{k+1}-1)(\lambda/2-\ell)}\}$.

\item The coefficients of $\omega^{\nu 2^{k+1}}$ for $\nu \in \mathbb{Z}$ (and in particular constant coefficients) of $\omega^{(\lambda/2)2^k} T^\ell(\omega^{2^k})\cdots \omega^{\lambda/2} T^\ell(\omega)$ are determined entirely by the coefficients of $\omega^{\nu 2^k}$ for $\nu \in \mathbb{Z}$ of $\omega^{(\lambda/2)2^{k-1}} T^\ell(\omega^{2^{k-1}})\cdots \omega^{\lambda/2} T^\ell(\omega)$.
\end{enumerate}

We continue to let $P$ be the space of Laurent polynomials in $\omega$ and define the operator $S^{\ell,\lambda}$ on $P^{2\ell+1}$ as follows: if for $A \in P^{2\ell+1}$,
$$
\omega^{\lambda/2} T^\ell(\omega) A = \begin{bmatrix} \sum_{j\in \mathbb{Z}} \beta_{-\ell}(j) \omega^j, & \cdots & ,\sum_{j \in \mathbb{Z}} \beta_\ell(j) \omega^j \end{bmatrix}^T,
$$
define
$$
S^{\ell,\lambda} A := \begin{bmatrix} \sum_{j \in \mathbb{Z}} \beta_{-\ell}(2j) \omega^j, & \cdots & ,\sum_{j \in \mathbb{Z}} \beta_\ell(2j) \omega^j \end{bmatrix}^T.
$$
Then just as before, for any $v \in \mathbb{C}^{2\ell+1}$, we have
$$
\mathbb{E} \,\omega^{(\lambda/2)2^k} T^\ell(\omega^{2^k})\cdots \omega^{\lambda/2} T^\ell(\omega) = \mathbb{E}\, (S^{\ell,\lambda})^{k+1}v.
$$
We define $P_{\ell,\lambda}:=\textrm{span}_\mathbb{C}\{\omega^{\lambda/2-\ell+1},...,\omega^{\lambda/2+\ell-1}\}$, and note that $S^{\ell,\lambda}$ maps $(P_{\ell,\lambda})^{2\ell+1}$ into itself, and moreover in this fourth case $\mathbb{C}^{2\ell+1} \subset (P_{\ell,\lambda})^{2\ell+1}$. Thus letting $S_{\ell,\lambda}$ be the operator $S^{\ell,\lambda}$ restricted to the finite dimensional complex vector space $(P_{\ell,\lambda})^{2\ell+1}$, if as before we demonstrate $\rho(S_{\ell,\lambda}) < 1$, we will be done. 

The same argument we used in the previous section shows that we will have such a bound on the spectral radius provided there is no nonzero $(2\ell+1)\times (2\ell-1)$ array of numbers $\alpha_h(j)$ such that
$$
\omega^{\lambda/2} \tau^\ell \begin{pmatrix} \omega^{-\ell} & & & \\ & \omega^{-(\ell-1)} & & & \\ & & \ddots & \\ & & & \omega^\ell \end{pmatrix} \begin{pmatrix} \sum_j \alpha_{-\ell}(j) \omega^j \\ \vdots \\ \sum_j \alpha_\ell(j) \omega^j \end{pmatrix} = c \begin{pmatrix} \sum_j \alpha_{-\ell}(j) \omega^{2j} \\ \vdots \\ \sum_j \alpha_\ell(j) \omega^{2j} \end{pmatrix}
$$
for a constant $|c|=1$, where in each sum $j$ runs from $\lambda/2-\ell+1$ to $\lambda/2+\ell-1$.

Let $\mathcal{L}(\nu)$ be the linear equation obtained from examining the coefficient of $\omega^\nu$. The argument will be the same as the previous section except now there will be no need to examine an extraneous middle column. We first consider $\omega$ taken to an odd power, that is the equations $\mathcal{L}(\lambda-2\ell+1), \mathcal{L}(\lambda-2\ell+3),...,\mathcal{L}(\lambda+2\ell-1)$, to see that the array $(\alpha_h(j))$ has entries all $0$ across alternating skew diagonals. We then turn to even powers, considering in succession $\mathcal{L}(\lambda-2\ell+2), \mathcal{L}(\lambda-2\ell+4),...,\mathcal{L}(\lambda-1)$, showing that the first $(2\ell+1)/2$ columns consist of $0$ entries by using Proposition \ref{prop:special_matrix2}. Finally consider $\mathcal{L}(\lambda+2\ell-2),...,\mathcal{L}(\lambda+1)$ to see the same for the last $(2\ell+1)/2$ columns. This verifies that indeed all entries of any such array $(\alpha_h(j))$ must be $0$. Hence $\rho(S_{\ell,\lambda}) < 1$ and \eqref{eq:tends_to_0} is true, which completes the analysis of case 2.
	
What remains is \textbf{case 1}. This is much the same argument and we leave the details to the reader. Again we must break the proof up into cases depending upon whether $0$ lies outside, on the boundary of, or inside the interval $[\lambda/2-\ell,\lambda/2+\ell]$. If $0$ lies outside, matters are again trivial. In the case that $0$ lies on the boundary, we need the additional computation that for $\ell = 1,2,...$, we have $|\tau^\ell_{-\ell,-\ell}| < 1$ and $|\tau^\ell_{\ell,\ell}| < 1,$ but use the same argument. Analysis of the case that $0$ lies inside the interval does not substantially differ from that in section \ref{sec:1}, and in particular all necessary facts about the matrix $\tau^\ell$ have already been proved there.

\section{A remark on equidistribution}

The proofs we have given of the equidstribution results Theorems \ref{thm:matrix_product1} and \ref{thm:matrix_product2} have depended upon the special form of the matrices $g(\omega)$ and $G(\omega)$. It is natural to ask whether this need be so. Indeed, by the analogy with random walks made in the introduction, one may be led to believe more generally that for a compact group $H$ and a function $f: \mathbb{R}/\mathbb{Z} \rightarrow H$ such that $f$ is supported on neither a proper closed subgroup ('non-degeneracy') nor a coset of a proper closed subgroup ('aperiodicity'), the product
\begin{equation}
\label{eq:group_product}
f(2^k t) f(2^{k-1} t)\cdots f(t),
\end{equation}
will equidistribute on $H$ as $k\rightarrow \infty$, where $t \in \mathbb{R}/\mathbb{Z}$ is a random variable with uniform distribution. (More generally one may think to replace multiplication by $2$ by another ergodic or mixing action on a probability space.)

Theorems \ref{thm:matrix_product1} and \ref{thm:matrix_product2} furnish examples of such a phenomenon, but it does not hold in general. Indeed for $H = \mathbb{Z}/2\mathbb{Z}$, let
$$
f(t) = \begin{cases} 0 & \mathrm{if} \; t\in [0,\tfrac18)\cup [\tfrac12,\tfrac58)\cup [\tfrac34,1) \\ 1 & \mathrm{if}\; t\in[\tfrac18,\tfrac12)\cup [\tfrac58,\tfrac34) \end{cases}
$$
and extend $f$ periodically. Then for $k\geq 1$, one may see that 
$$
\mathbb{P}( f(2^k t)\cdots f(t) = 0) = 5/8,
$$
which certainly does not tend to $1/2$.

It would be interesting nonetheless to understand better what general conditions on $f$ ensure that products like \eqref{eq:group_product} equidistribute. Such products bear a resemblance to those arising in certain noncommutative ergodic theorems \cite{KaLe}, but involve a different sort of averaging at a different scale and therefore seem to require different techniques.

\section{Acknowledgements}
I thank Hugh Montgomery for a number of helpful discussions regarding this problem along with making his notes on Rudin-Shapiro polynomials available to me and also Anders Karlsson, Jeff Lagarias, Alon Nishry, and Doron Zeilberger for helpful feedback.


\begin{thebibliography}{99}

\bibitem{AlSh} J.-P. Allouche and J. Shallit. \emph{Automatic Sequences: Theory, Applications, Generalizations.} Cambridge University Press (2003).

\bibitem{AnAsRo} G. E. Andrews, R. Askey, and R. Roy. \emph{Special functions.} Encyclopedia of Mathematics and its Applications, 71. Cambridge University Press (1999).

\bibitem{Ap} D. Applebaum. \emph{Probability on compact Lie groups.}  Probability Theory and Stochastic Modelling, 70. Springer (2014).

\bibitem{Be} J. Beck. ``Flat polynomials on the unit circle - note on a problem of Littlewood." \emph{Bull. London Math. Soc.} 23 (1991): 269-277.



\bibitem{Bri} J. Brillhart. ``On the Rudin-Shapiro polynomials." \emph{Duke Math. J.} 40.2 (1973): 335-353.

\bibitem{BrLoMo} J. Brillhart, J.S. Lomont, and P. Morton. ``Cyclotomic properties of the Rudin-Shapiro polynomials." \emph{J. Reine Angew. Math.} 288 (1976): 37-65. 

\bibitem{BoBo} E. Bombieri and J. Bourgain. ``On Kahane's ultraflat polynomials." \emph{J. Eur. Math. Soc.} 11 (2009): 627-703.

\bibitem{DoHa} C. Doche and L. Habsieger, ``Moments of the Rudin-Shapiro polynomials." \emph{J. Fourier Anal. Appl.} 10 (2004).

\bibitem{Ze} S.B. Ekhad and D. Zeilberger, ``Integrals Involving Rudin-Shapiro Polynomials and Sketch of a Proof of Saffari's Conjecture." \emph{Number Theory: In Honor of Krishna Alladi’s 60-th Birthday}, G. Andrews and F. Gravan, eds., Springer, to appear. Available at \url{http://arxiv.org/abs/1605.06679} and \url{http://www.math.rutgers.edu/~zeilberg/mamarim/mamarimhtml/hss.html} .

\bibitem{Er} T. Erd\'elyi, ``The Mahler Measure of the Rudin-Shapiro Polynomials." \emph{Constructive Approximation.} 43 (2016) 357-369.

\bibitem{Er2} T. Erd\'elyi, ``The phase problem of ultraflat unimodular polynomials: the resolution of the conjecture of Saffari." \emph{Math. Ann.} 321.4 (2001): 905-924.

\bibitem{Fa} J. Faraut, \emph{Analysis on Lie Groups: An Introduction.} Cambridge Studies in Advanced Mathematics, 110. Cambridge University Press (2008).
	
\bibitem{Go} M. Golay, ``Multislit spectrometry." \emph{J. Opt. Soc. Amer.} 39 (1949): 437-444.

\bibitem{Ka} J.-P. Kahane, ``Sur les polyn\^omes \`a coefficients unimodulaires." \emph{Bull. London Math. Soc.} 12 (1980): 321–342.

\bibitem{KaSa} J.-P. Kahane and R. Salem, \emph{Ensembles parfaits et s\'eries trigonom\'etriques.} Hermann, Revised Edition (1994).

\bibitem{KaLe} A. Karlsson and F. Ledrappier, ``Noncommutative  ergodic  theorems." \emph{Geometry, rigidity, and group actions.} Chicago Lectures in Math., Univ. Chicago Press (2011): 396 - 418.

\bibitem{KaIt} Y. Kawada and K. It\^o, ``On the probability distribution on a compact group." \emph{I. Proc. Phys.-Math. Soc. Japan.} 22 (1940) 977-998.

\bibitem{laCo} A. la Cour-Harbo, ``On the Rudin-Shapiro transform." \emph{Appl. Comput. Harmon. Anal.} 24.3 (2008): 310–328.

\bibitem{Mo} H. L. Montgomery, ``Littlewood polynomials.” \emph{Number Theory: In Honor of Krishna Alladi’s 60-th Birthday}, G. Andrews and F. Gravan, eds., Springer, to appear.

\bibitem{Mo2} H. Montgomery, in ``Problem  Session." \emph{Mathematisches  Forschungsinstitut Oberwolfach, Report No. 51/2013} (2013): 3029-3030.

\bibitem{Od} A. M. Odlyzko, ``Search for ultraflat polynomials with plus and minus one coefficients." \emph{Preprint.} Available at \url{http://www.dtc.umn.edu/~odlyzko/doc/ultraflat.pdf} .

\bibitem{Ru} W. Rudin, ``Some theorems on Fourier coefficients," \emph{Proc. Amer. Math. Soc.} 10 (1959) 855-859.

\bibitem{SaZy} R. Salem, and A. Zygmund, ``On lacunary trigonometric series." \emph{Proc. Nat. Acad. Sci. USA.} 33 (1947) 333-338.

\bibitem{Sh} H. S. Shapiro, \emph{Extremal problems for polynomials.} M. S. Thesis, MIT (1951).

\bibitem{Vi} N. Vilenkin, \emph{Special Functions and the Theory of Group Representations.} Translations of Mathematical Monographs. American Mathematical Society (1968).

	
\end{thebibliography}
\end{document}